\documentclass[a4paper,10pt]{amsart}%
\usepackage{eurosym}
\usepackage{amsmath}
\usepackage{mathrsfs}
\usepackage{amsfonts}
\usepackage{graphicx}
\usepackage{color}
\usepackage{amsfonts}
\usepackage{amssymb}%
\usepackage{esint}

\usepackage{palatino}

\usepackage[abbrev]{amsrefs}
\usepackage{mathtools}
\usepackage{tikz-cd}
\usepackage{float}

\newtheorem{theorem}{Theorem}[section]

\newtheorem{openquestion}[theorem]{Open Question}
\newtheorem{corollary}[theorem]{Corollary}

\newtheorem{definition}[theorem]{Definition}

\newtheorem{lemma}[theorem]{Lemma}

\newtheorem{proposition}[theorem]{Proposition}
\newtheorem{remark}[theorem]{Remark}

\numberwithin{equation}{section}

\newcommand{\barint}{
\rule[.036in]{.12in}{.009in}\kern-.16in \displaystyle\int }

\newcommand{\barcal}{\mbox{$ \rule[.036in]{.11in}{.007in}\kern-.128in\int $}}

\DeclareMathOperator*{\argmin}{arg\,min}

\usepackage{enumitem}
\makeatletter
\let\@wraptoccontribs\wraptoccontribs
\makeatother

\mathchardef\mhyphen="2D

\title{On functions of bounded $\beta$-dimensional mean oscillation}

\author[Y.-W. Chen]{You-Wei Benson Chen}
\author[D. Spector]{Daniel Spector}

\address[Y.-W. Chen]{
National Chiao Tung University, Department of Applied Mathematics, 1001 Ta Hsueh Rd, 30010 Hsinchu, Taiwan, R.O.C.
}
\email{bensonchen.sc07@nycu.edu.tw}

\address[D. Spector]{Department of Mathematics, National Taiwan Normal University, No. 88, Section 4, Tingzhou Road, Wenshan District, Taipei City, Taiwan 116, R.O.C.
}
\email{spectda@protonmail.com}

\date{}

\begin{document}

\maketitle

\begin{abstract}
In this paper, we define a notion of $\beta$-dimensional mean oscillation of functions $u: Q_0 \subset \mathbb{R}^d \to \mathbb{R}$ which are integrable on $\beta$-dimensional subsets of the cube $Q_0$:
\begin{align*}
\|u\|_{BMO^{\beta}(Q_0)}\vcentcolon=   \sup_{Q \subset Q_0} \inf_{c \in \mathbb{R}} \frac{1}{l(Q)^\beta}  \int_{Q}  |u-c| \;d\mathcal{H}^{\beta}_\infty,
\end{align*}
where the supremum is taken over all finite subcubes $Q$ parallel to $Q_0$, $l(Q)$ is the length of the side of the cube $Q$, and $\mathcal{H}^{\beta}_\infty$ is the Hausdorff content.  In the case $\beta=d$ we show this definition is equivalent to the classical notion of John and Nirenberg, while our main result is that for every $\beta\in (0,d]$ one has a dimensionally appropriate analogue of the John-Nirenberg inequality for functions with bounded $\beta$-dimensional mean oscillation:  There exist constants $c,C>0$ such that
\begin{align*}
\mathcal{H}^{\beta}_\infty \left(\{x\in Q:|u(x)-c_Q|>t\}\right) \leq C l(Q)^\beta \exp(-ct/\|u\|_{BMO^\beta(Q_0)})
\end{align*}
for every $t>0$, $u \in BMO^\beta(Q_0)$, $Q\subset Q_0$, and suitable $c_Q \in \mathbb{R}$.  Our proof relies on the establishment of capacitary analogues of standard results in integration theory that may be of independent interest.
 
\end{abstract}

\section{Introduction}
A fundamental result of F. John and L. Nirenberg \cite{JN} asserts that if $Q_0$ is a cube in $\mathbb{R}^d$ and $u \in L^1(Q_0)$ satisfies
\begin{align}\label{bmo_cube}
\sup_{Q \subset Q_0} \frac{1}{|Q|}\int_Q |u-u_Q|\;dx <+\infty,
\end{align}
then the integrand in \eqref{bmo_cube} enjoys an exponential decay estimate of its level sets: There exist constants $c,C>0$ such that
\begin{align}\label{jninequality}
|\{x\in Q:|u(x)-u_Q|>t\}|\leq C|Q| \exp\left(-ct\right)
\end{align}
for all $Q\subset Q_0$.  Here we use $u_Q$ to denote the average value of $u$ over $Q$ and all cubes are assumed to be parallel to $Q_0$.  That \eqref{jninequality} implies \eqref{bmo_cube} is an easy consequence of expressing the Lebesgue integral as integration over the level sets, so that somewhat surprisingly \eqref{bmo_cube} and \eqref{jninequality} are equivalent in characterizing what was subsequently termed the space of functions of bounded mean oscillation in $Q_0$:
\begin{align*}
BMO(Q_0)\vcentcolon= \left\{ u \in L^1(Q_0) :
\|u\|_{BMO(Q_0)}\vcentcolon= \sup_{Q \subset Q_0}  \frac{1}{|Q|} \int_Q |u-u_Q|\;dx<+\infty\right\}.
\end{align*}


Typical examples of functions of bounded mean oscillation include bounded functions, $f \ast \log|\cdot|$ for any $f \in L^1(\mathbb{R}^d)$, and functions in critical Sobolev spaces, though the sense in which any of these examples is generic is itself an interesting question.  Clearly bounded functions admit a significant improvement to \eqref{jninequality}, while the latter examples have special properties of their Fefferman-Stein decomposition (originally established in \cite{Fefferman,FeffermanStein}, with a constructive proof given in \cite{Uchiyama}, special properties demonstrated in \cite{BourgainBrezis2002,GS,GS1}, and discussed further in \cite{Spector-ND}) as well as the well-definedness of restrictions on lower dimensional sets of arbitrarily small dimension with exponential decay of the Hausdorff content of their level sets \cite{Adams1973, MS}.  

The starting point of this paper is the consideration of the exponential decay of the Hausdorff content established in \cite{MS} alongside the John-Nirenberg result on the equivalence of \eqref{bmo_cube} and \eqref{jninequality}.  To illustrate the idea let us consider a simple example, that of $W_0^{1,d}(Q_0)$.  In particular, if $u \in W_0^{1,d}(Q_0)$ with $\|\nabla u\|_{L^d(Q_0)}\leq 1$, it is a consequence of the results in \cite{MS} that such $u$ admits an exponential decay of the $\beta$-dimensional spherical Hausdorff content $\mathcal{H}^\beta_\infty$ of its level sets for any $\beta \in (0,d]$:
\begin{align}\label{Martinez-Spector}
\mathcal{H}^\beta_\infty\left(\{x\in Q_0:|u(x)|>t\}\right) \leq C_\beta \exp(-c_\beta t^{d'})
\end{align}
for some $c_\beta,C_\beta>0$ and all $t>0$, where we recall that
\begin{align*}
\mathcal{H}^{\beta}_{\infty}(E)\vcentcolon= \inf \left\{\sum_{i=1}^\infty \omega_\beta r_i^\beta : E \subset \bigcup_{i=1}^\infty B(x_i,r_i) \right\}
\end{align*}
and $\omega_\beta\vcentcolon= \pi^{\beta/2}/\Gamma(\beta/2+1)$ is a normalization constant.  One is tempted to compare the inequalities \eqref{jninequality} and \eqref{Martinez-Spector}, though as only the former is scale invariant it is more appropriate to find a replacement for the latter which satisfies this property, which for the moment we refer to as \eqref{Martinez-Spector-prime}.  In the identification of the existence of such an inequality the idea of the paper emerges:  The equivalence of \eqref{bmo_cube} and \eqref{jninequality} and the parallel structure of \eqref{jninequality} and this \eqref{Martinez-Spector-prime} prompts one to wonder whether there is a counterpart of \eqref{bmo_cube} involving the $\beta$-dimensional Hausdorff content which is equivalent to \eqref{Martinez-Spector-prime}.  

The main contribution of this paper is the pursuit of these ideas, and in particular, to give a definition of this counterpart \eqref{nondyadic_norm}, the $\beta$-dimensional mean oscillation, such that functions for which this quantity is bounded admit suitable analogues to properties known in the classical theory of functions of bounded mean oscillation.  The relationships we identify can be represented schematically as 
\begin{figure}[H]
\centering
\begin{tikzcd} 
	{\;(1.1)\;} & {\;(1.2)\;} \\
	{\;(1.1')} & {\;(1.2'),}
	\arrow[from=1-1, to=1-2, Leftrightarrow]
	\arrow[from=2-1, to=2-2, Leftrightarrow]
	\arrow[dashed, no head, from=1-2, to=2-2]
	\arrow[dashed, no head, from=1-1, to=2-1]
\end{tikzcd}
\end{figure}
\noindent
while our results comprise the discovery of the existence and equivalence of the bottom portion of the diagram.  A posteriori one finds that \eqref{Martinez-Spector-prime} is given by
\begin{align}\label{Martinez-Spector-prime}
\mathcal{H}^\beta_\infty\left(\{x\in Q:|u(x)-c_Q|>t\}\right) \leq C'_\beta l(Q)^\beta \exp(-c'_\beta t) \tag{\ref{jninequality}$'$}
\end{align}
for some $c'_\beta,C'_\beta>0$, all $t>0$, all $Q \subset Q_0$ parallel to $Q_0$, and suitable $c_Q \in \mathbb{R}$, while to define \eqref{nondyadic_norm}, we first recall a few relevant notions related to $\beta$-dimensional integrability of functions $f : Q_0 \subset \mathbb{R}^d \to \mathbb{R}$.  We say that a function $f$ is $\mathcal{H}^{\beta}_\infty$-quasicontinuous if for every $\epsilon>0$, there exists an open set $O$ such that $\mathcal{H}^{\beta}_\infty(O)<\epsilon$ and $f|_{O^c}$ is continuous.  For a non-negative $\mathcal{H}^{\beta}_\infty$-quasicontinuous function $f$, we define the Choquet integral of $f$ in $Q_0$ via the formula
\begin{align}\label{choquet}
\int_{Q_0} f\;d\mathcal{H}^{\beta}_\infty\vcentcolon= \int_0^\infty \mathcal{H}^{\beta}_\infty\left(\left\{x \in Q_0: f(x)>t\right\}\right)\;dt,
\end{align}
which enables us to define, for a signed $\mathcal{H}^{\beta}_\infty$-quasicontinuous function $f$,
\begin{align}\label{norm}
\|f\|_{L^1(Q_0;\mathcal{H}^{\beta}_\infty)}\vcentcolon= \int_{Q_0} |f|\;d\mathcal{H}^{\beta}_\infty.
\end{align}
With this preparation we now introduce the vector space
\begin{align*}
L^1(Q_0;\mathcal{H}^{\beta}_\infty)\vcentcolon= \left\{ f\; \mathcal{H}^{\beta}_\infty\text{-quasicontinuous } : \|f\|_{L^1(Q_0;\mathcal{H}^{\beta}_\infty)}<+\infty\right\}.
\end{align*}
It can be shown that \eqref{norm} is a quasi-norm, that there exists a norm equivalent to \eqref{norm} such that $L^1(Q_0;\mathcal{H}^{\beta}_\infty)$ equipped with this equivalent norm is a Banach space, that the space of continuous and bounded functions in $Q_0$ for which this equivalent norm (or alternatively \eqref{norm}) is finite is dense in this Banach space, and that the topological dual of this space can be identified with the Morrey space
\begin{align*}
\mathcal{M}^{\beta}(Q_0)\vcentcolon= \left\{ \nu \in M_{loc}(Q_0): \|\nu\|_{\mathcal{M}^\beta} \vcentcolon= \sup_{x \in Q_0,r>0} \frac{|\nu|(B(x,r))}{\;r^\beta}<\infty \right\},
\end{align*}
where $M_{loc}(Q_0)$ is the set of locally finite Radon measures in $Q_0$ (and extended by zero outside $Q_0$), see Section \ref{preliminaries} below.  For the present we only mention one further consequence of these facts, which follows from this duality, the validity of the formula
\begin{align}\label{HB}
\int_{Q_0} |f|\;d\mathcal{H}^{\beta}_\infty \asymp \sup_{\|\nu\|_{\mathcal{M}^\beta} \leq 1}  \left|\int_{Q_0} f \;d\nu \right|
\end{align}
for $f \in L^1(Q_0;\mathcal{H}^{\beta}_\infty)$.

 We are now prepared to define \eqref{nondyadic_norm}:  For a function $u \in L^1(Q_0;\mathcal{H}^{\beta}_\infty)$, we say that $u$ has bounded $\beta$-dimensional mean oscillation if
\begin{align}\label{nondyadic_norm}
 \sup_{Q \subset Q_0} \inf_{c \in \mathbb{R}} \frac{1}{l(Q)^\beta}  \int_{Q}  |u-c| \;d\mathcal{H}^{\beta}_\infty<+\infty,\tag{\ref{bmo_cube}$'$}
\end{align}
where the supremum is taken over all finite subcubes $Q \subset Q_0$ parallel to $Q_0$.  We then define the space of functions of bounded $\beta$-dimensional mean oscillation in $Q_0$:
\begin{align*}
BMO^{\beta}(Q_0)\vcentcolon=  \left\{ u \in L^1(Q_0;\mathcal{H}^{\beta}_\infty) :  \|u\|_{BMO^{\beta}(Q_0)}<+\infty\right\},
\end{align*}
where
\begin{align*}
\|u\|_{BMO^{\beta}(Q_0)} \vcentcolon=\sup_{Q \subset Q_0} \inf_{c \in \mathbb{R}} \frac{1}{l(Q)^\beta}  \int_{Q}  |u-c| \;d\mathcal{H}^{\beta}_\infty.
\end{align*}

As an almost immediate consequences of the definition, let us record two theorems of interest.  First, the space $BMO^\beta(Q_0)$ is closed under composition with Lipschitz functions, a result we record in
\begin{theorem}\label{composition}
Let $\beta \in (0,d]$ and $1\leq p<+\infty$.  If $u \in BMO^\beta(Q_0)$ and $\phi \in \operatorname*{Lip}(\mathbb{R};\mathbb{R})$ with $\phi(0)=0$, then $\phi \circ u \in BMO^\beta(Q_0)$ and 
\begin{align*}
\|\phi \circ u \|_{BMO^{\beta}(Q_0)} \leq \operatorname*{Lip}(\phi) \|u\|_{BMO^{\beta}(Q_0)}.
\end{align*}
\end{theorem}
Second, by the functional property \eqref{HB}, for $\beta=k \in \mathbb{N}$, such functions restrict to $k$-dimensional hyperplanes and are of bounded mean oscillation on these subspaces, a result we record in
\begin{theorem}\label{restriction}
Let $k \in \mathbb{N} \cap [1,d]$.  If $u \in BMO^k(Q_0)$, then for any $k$-dimensional hyperplane $H_k$ parallel to $Q_0$, $ u|_{Q_0\cap H_k} \in BMO(Q_0\cap H_k)$ and
\begin{align*}
\|u|_{Q_0\cap H_k}\|_{BMO(Q_0\cap H_k)} \leq C\|u\|_{BMO^k(Q_0)}.
\end{align*}
In this case $k=d$, one has
\begin{align*}
 \frac{1}{C} \|u\|_{BMO^d(Q_0)} \leq \|u\|_{BMO(Q_0)} \leq C \|u\|_{BMO^d(Q_0)} 
\end{align*}
so that $BMO^d(Q_0) \equiv BMO(Q_0)$.
\end{theorem}

Theorem \ref{restriction} and the classical argument of John and Nirenberg together imply that $u \in BMO^k(Q_0)$ satisfies \eqref{jninequality} on $k$-hyperplanes.  Our main result is that for any $\beta \in (0,d]$, and for any $\beta$-dimensional set, these spaces satisfy an analogue of \eqref{jninequality}:
\begin{theorem}\label{jn_content}
Let $\beta \in (0,d]$.  There exist constants $c,C>0$ such that
\begin{align}\label{jninequality_content_prime}
\mathcal{H}^{\beta}_\infty\left(\{x\in Q:|u(x)-c_Q|>t\}\right) \leq C l(Q)^\beta \exp(-ct/\|u\|_{BMO^\beta(Q_0)})
\end{align}
for every $t>0$, $u \in BMO^\beta(Q_0)$, all finite subcubes $Q\subset Q_0$ parallel to $Q_0$, and suitable $c_Q \in \mathbb{R}$.
\end{theorem}

Our proof of Theorem \ref{jn_content} follows the argument of John and Nirenberg, with suitable modifications necessitated by the replacement of the Lebesgue measure with the Hausdorff content, which is not a measure but only an outer measure with some additional properties.  More precisely, the actual arguments in our proof are via a more convenient, though equivalent (see Section \ref{preliminaries}), quantity, the dyadic Hausdorff content adapted to any given cube $Q \subset Q_0$.  This requires several results for this object that may be of independent interest, including a weak form of the Lebesgue differentiation theorem we prove in Corollary \ref{differentialhausdorfflimsup} and a Calder\'on-Zygmund decomposition we establish in Theorem \ref{cz}, see Section \ref{content-jn} for further details.  These results parallel the standard ones for measures and make use of the fact that the Hausdorff content is in a certain sense doubling.  There are some difficulties, however, due to the nonlinearity of the Choquet integral, which we handle through the use of an additional covering lemma, see Proposition \ref{ov} in Section \ref{preliminaries} below.

Theorem \ref{jn_content} has a number of interesting consequences.  First, a standard application of the inequality \eqref{jninequality_content_prime} is an exponential integrability result.  In our setting, we have the following corollary which asserts that functions in $BMO^\beta(Q_0)$ are exponentially integrable on sets of dimension $\beta$.
\begin{corollary}\label{exp_integrability}
Let $\beta \in (0,d]$.  There exist constants $c',C'>0$ such that
\begin{align*}
\int_Q \exp\left(c'|u-c_Q|\right) \;d\mathcal{H}^{\beta}_\infty \leq C'l(Q)^\beta
\end{align*}
for every $u \in BMO^\beta(Q_0)$ with $\|u\|_{BMO^\beta(Q_0)}\leq 1$ and for all finite subcubes $Q\subset Q_0$ parallel to $Q$ and suitable $c_Q$.
\end{corollary}
Second, a classical result for the space of functions of bounded mean oscillation is the equivalence of a family of semi-norms indexed by $p \in [1,\infty)$.  A similar result for the space of functions of bounded $\beta$-dimensional mean oscillation follows easily from Theorem \ref{jn_content}  or Corollary \ref{exp_integrability}.  That is, if we introduce
\begin{align*}
\|u\|_{BMO^{\beta,p}(Q_0)}\vcentcolon=  \sup_{Q \subset Q_0} \inf_{c \in \mathbb{R}} \left(\frac{1}{l(Q)^\beta} \int_{Q}  |u-c|^p d\mathcal{H}^{\beta}_\infty\right)^{1/p},
\end{align*}
we will show
\begin{corollary} \label{pseminorms}
Let $\beta \in (0,d]$.  There exists a constant $C=C(\beta)>0$ such that
\begin{align*}
\frac{1}{C}\|u\|_{BMO^{\beta}(Q_0)} \leq \|u\|_{BMO^{\beta,p}(Q_0)} \leq C p\|u\|_{BMO^{\beta}(Q_0)}
\end{align*}
for all $u \in BMO^\beta(Q_0)$.
\end{corollary}

Third, from Theorem \ref{jn_content} one obtains a nesting of the spaces, a result we record in
\begin{corollary}\label{nesting}
Let $0<\alpha \leq \beta \leq d$ and suppose $u \in BMO^\alpha(Q_0)$.  Then $u \in BMO^\beta(Q_0)$ and
\begin{align*}
\|u\|_{BMO^\beta(Q_0)} \leq C\|u\|_{BMO^\alpha(Q_0)}
\end{align*}
for a numerical constant $C=C(\alpha,\beta)>0$ independent of $u$.
\end{corollary}

\begin{remark}
The preceding results imply that functions in $BMO^\alpha(Q_0)$ are exponentially integrable on sets of every dimension $\beta \in [\alpha,d]$ and that for $k \in \mathbb{N}$, $k \geq \beta$, such functions restrict to $k$-dimensional hyperplanes and are of bounded mean oscillation in these hyperplanes with respect to the natural measure.  This gives another perspective of the question of restriction of functions of bounded mean oscillation addressed in \cite{Nakai,KSS}.
\end{remark}

Finally, we return to the impetus for the initial consideration of these quantities, the question of improvements to critical Sobolev embeddings studied in \cite{MS}.  An interesting case which was not treated there is the Morrey space $\mathcal{M}^{d-\alpha}(\mathbb{R}^d)$.  This space contains $L^{d/\alpha,\infty}(\mathbb{R}^d)$ (the weak $L^{d/\alpha}(\mathbb{R}^d)$ space), so that one can obtain no better than the linear exponential decay with respect to the Hausdorff content of suitable dimension obtained for this space in \cite{MS}.  That one has this linear exponential decay follows from Theorem \ref{jn_content} and the following result, which refines D.R. Adams' $BMO$ estimate in \cite{Adams1975}.
\begin{theorem}\label{morrey}
For $\alpha \in (0,d)$ and $\epsilon \in(0, \alpha]$, there exists a constant $C=C(\alpha,\epsilon)>0$ such that
\begin{align*}
\|I_\alpha \mu \|_{BMO^{d-\alpha+\epsilon}(\mathbb{R}^d)} \leq C \|\mu\|_{\mathcal{M}^{d-\alpha}(\mathbb{R}^d)}
\end{align*}
for all $\mu \in \mathcal{M}^{d-\alpha}(\mathbb{R}^d)$ such that $I_\alpha \mu \in L^1(K,\mathcal{H}^{d-\alpha+\epsilon}_\infty)$ for all compact sets $K \subset \mathbb{R}^d$.
\end{theorem}
\noindent
Here we recall the definition of the Riesz potential of order $\alpha \in (0,d)$ of a measure $\mu$ is
\begin{align*}
I_\alpha \mu(x)\vcentcolon= \frac{1}{\gamma(\alpha)} \int_{\mathbb{R}^d} \frac{d\mu(y)}{|x-y|^{d-\alpha}}
\end{align*}
where $\gamma(\alpha)$ is a normalization constant \cite[p.~117]{S}.   

The local capacitary integrability assumption $I_\alpha \mu \in L^1(K,\mathcal{H}^{d-\alpha+\epsilon}_\infty)$ for all compact sets $K \subset \mathbb{R}^d$ is analogous to D.R. Adams' local Lebesgue integrability assumption \cite[Proposition 3.3 on p.~770]{Adams1975}, without which it is clear that one cannot obtain $I_\alpha \mu \in BMO^{d-\alpha+\epsilon}(\mathbb{R}^d)$, as this space is locally is a subspace of  $L^1(K,\mathcal{H}^{d-\alpha+\epsilon}_\infty)$.  However, one wonders whether it is necessary to assume this or if it can be obtained from the other hypotheses, which we here ask as
\begin{openquestion}
Suppose $\mu \in \mathcal{M}^{d-\alpha}(\mathbb{R}^d)$ and $\epsilon \in (0,\alpha]$.  Is it true that $I_\alpha \mu \in L^1(K,\mathcal{H}^{d-\alpha+\epsilon}_\infty)$ for all compact sets $K \subset \mathbb{R}^d$?
\end{openquestion}
\noindent
The answer is no when $\epsilon=0$, as we demonstrate in the following examples.  First, in the simpler case $\alpha=k \in \mathbb{N}$, with the choice $\mu=\mathcal{H}^{d-k}|_S$ for a portion of a $d-k$-dimensional hypersurface $S$, one finds that $|I_k \mu|$ behaves locally like $|\log(d(x,S))|$ where $d(x,S)$ is the distance to the surface (see \cite{Spector-PM} for the case $\alpha=1$ and \cite{RSS} for the case $\alpha=d-1$, analogous estimates can be made for the other cases).  In these cases $|I_k \mu|=+\infty$ on the $d-k$-dimensional hypersurface $S$ and therefore $I_\alpha \mu \notin L^1(S,\mathcal{H}^{d-k}_\infty)$.  When $\alpha$ is not an integer, the argument is slightly more involved.  In particular, fixing $d-\alpha \notin \mathbb{N}$ one can apply a construction of J. Hutchinson \cite{Hutchinson} to generate a $d-\alpha$ regular measure $\mu$ supported in a cube $Q_0$, i.e. $\mu$ which satisfies
\begin{align}\label{measure_density}
c \leq \liminf_{l(Q) \to 0} \frac{\mu(Q)}{l(Q)^{d-\alpha}} \leq \limsup_{l(Q) \to 0} \frac{\mu(Q)}{l(Q)^{d-\alpha}} \leq C
\end{align}
for cubes $Q$ centered at $x$, for $\mu$ almost every $x \in \mathbb{R}^d$.  We claim $I_\alpha \mu \notin L^1(Q_0;\mathcal{H}^{d-\alpha}_\infty)$.  Indeed, if this were the case one would have
\begin{align*}
+\infty>\int_{Q_0} |I_\alpha \mu| d\mathcal{H}^{d-\alpha}_\infty &\geq c\int_{Q_0} I_\alpha \mu(x) d\mu(x) \\
&= c'\int_{Q_0} \int_{\mathbb{R}^d} \frac{1}{|x-y|^{d-\alpha}} d\mu(y) d\mu(x).
\end{align*}
However, this would imply that the $(d-\alpha)$-dimensional Riesz transform
\begin{align*}
 \lim_{\epsilon \to 0}\int_{\mathbb{R}^d \setminus B(x,\epsilon)} \frac{x-y}{|x-y|^{d-\alpha+1}} d\mu(y)
\end{align*}
exists and is finite for $\mu$ almost every $x \in \mathbb{R}^d$, which along with the upper and lower density conditions \eqref{measure_density}, would yield by \cite[Theorem 4.5]{MP} that $d-\alpha$ is an integer, a contradiction.  Therefore $I_\alpha \mu \notin L^1(Q_0;\mathcal{H}^{d-\alpha}_\infty)$.

Thus in general one cannot expect $I_\alpha \mu \in L^1(Q_0;\mathcal{H}^{d-\alpha}_\infty)$ for $\mu \in \mathcal{M}^{d-\alpha}(\mathbb{R}^d)$.  This does not rule out the possibility that $\mu$ which happen to satisfy this property might enjoy a similar conclusion to that of Theorem \ref{morrey}, which is our
\begin{openquestion}
If one assumes that $\mu \in \mathcal{M}^{d-\alpha}(\mathbb{R}^d)$ and $I_\alpha \mu \in L^1(K,\mathcal{H}^{d-\alpha}_\infty)$ for all compact sets $K \subset \mathbb{R}^d$, does one have $I_\alpha \mu \in BMO^{d-\alpha}(\mathbb{R}^d)$?
\end{openquestion}

Before discussing the plan of the paper, a few remarks are in order on the study of these spaces $BMO^\beta(Q_0)$ and their relationship with the results of Jean Van Schaftingen \cite{VS}.  The observation that there are spaces of interest between critical Sobolev spaces and $BMO$ is the subject of Jean Van Schaftingen's paper \cite{VS}.  In particular, he there utilizes endpoint $L^1$ inequalities pioneered by J. Bourgain and H. Brezis in \cite{BourgainBrezis2004,BourgainBrezis2007}, and to which he contributed \cite{VS,VS1,VS2, VS3}, as a starting place to define such spaces.  The spaces he studies are the dual of closed $k$-forms with integrable coefficients for $k=1,\ldots, d-1$, a discrete family of spaces for which he establishes various properties in a similar spirit to what we record here, including nesting in his Theorem 3.1, extension in his Theorem 3.2, restriction in his Theorem 3.4, and exponential integrability in his Proposition 5.6.  For Van Schaftingen's spaces, there are still a number of open questions, including whether they are closed under composition with Lipschitz functions and admit restriction/exponential integrability up to the endpoint (see the end of Section 5 on p.~18 and Section 6 in \cite{VS}), results which we can settle for the spaces defined in this paper.  It seems an interesting problem to compare the spaces he considers and the ones we define here for suitable choices of $k$ and $\beta$.


The plan of the paper is as follows.  In Section \ref{preliminaries}, we recall some necessary preliminaries related to the Hausdorff content and Choquet integrals.  In Section \ref{content-jn}, we establish results for the dyadic Hausdorff content adapted to an arbitrary cube $Q$ which are analogous to results known for the Lebesgue measure.  This includes a weak-type estimate for a dyadic Hausdorff content maximal function in Theorem \ref{max}, a weak form of the Lebesgue differentiation theorem in Corollary \ref{differentialhausdorfflimsup}, and a Calder\'on-Zygmund decomposition in Theorem \ref{cz}.  In Section \ref{mainresult}, we use the results up to this point in the paper to prove Theorem \ref{jn_content} and then from this theorem we deduce Corollaries \ref{exp_integrability}, \ref{pseminorms}, and \ref{nesting}.   Finally, in Section \ref{otherresults} we prove the remainder of the results asserted in the introduction, Theorems \ref{composition}, \ref{restriction}, and \ref{morrey}.

\section{Preliminaries}\label{preliminaries}
The purpose of this section is to collect some necessary preliminaries concerning the dyadic Hausdorff contents adapted to cubes, the Choquet integrals defined in terms of these contents, and the function spaces generated by these integrals.  We rely on results established in \cite{AdamsChoquet,AdamsChoquet1,OP,OP1,OV,ST,STW}, see also \cite{H,H1,PS, PS1, Saito, Saito1, STW1} for other results concerning Choquet integrals.  The subtle point is that these objects are not measures, so that integration with respect to them is a delicate matter.  We begin with some general definitions concerning set functions, and to this end we let $\mathcal{P}(\mathbb{R}^d)$ denote the class of all subsets $\mathbb{R}^d$.

\begin{definition}
We say that a set function $H:\mathcal{P}(\mathbb{R}^d) \to [0,\infty]$ is an outer measure if $H$ satisfies
\begin{enumerate}
\item  $H(\emptyset)=0$;
\item  If $E \subset F$, then $H(E)\leq H(F)$;
\item If $E \subset \cup_{i=1}^\infty E_i$ then 
\begin{align*}
H(E) &\leq \sum_{i=1}^\infty H(E_i).
\end{align*}
\end{enumerate}
\end{definition}

\begin{remark}
In the work of D.R. Adams \cite{AdamsChoquet,AdamsChoquet1}, outer measures are also referred to as capacities in the sense of N. Meyers.
\end{remark}

For such set functions, the monotonicity property allows us to define a notion of integration:
\begin{definition}
Given an outer measure $H$, we define the Choquet integral of a non-negative function $f$ over a cube $Q$ by
\begin{align}\label{choquet_def}
\int_Q f \;dH:= \int_0^\infty H(\{x \in Q : f(x)>t\}) \;dt,
\end{align}
where the right hand side is understood as the Lebesgue integral of the function
\begin{align*}
t \mapsto H(\{x \in Q : f(x)>t\}).
\end{align*}
\end{definition}

The Choquet integral is not linear, and not even sublinear, unless one assumes additional properties of $H$.  We here introduce two assumptions that have been used to establish sublinearity.
\begin{definition}
We say that a set function $H:\mathcal{P}(\mathbb{R}^d) \to [0,\infty]$ is strongly subadditive if $H$ satisfies
\begin{align*}H(E\cup F)+H(E\cap F) \leq H(E)+H(F).
\end{align*}
for all $E, F \subset \mathbb{R}^d$.
\end{definition}
and
\begin{definition}
We say that a set function $H:\mathcal{P}(\mathbb{R}^d) \to [0,\infty]$ is continuous from below if 
\begin{align*}
H(E)=\lim_{n\to \infty} H(E_n).
\end{align*}
for every increasing sequence of arbitrary sets $\{E_n\}$, where $E= \bigcup_{n=1}^\infty E_n$.
\end{definition}

If $H$ is a set function which satisfies the preceding assumptions, the Choquet integral of such a set function is sublinear and satisfies a monotone convergence theorem (see \cite[Proposition 3.2]{STW}):
\begin{proposition}\label{sublinear}
Suppose $H$ is an outer measure which is strongly subadditive and continuous from below.  Then the Choquet integral with respect to $H$ satisfies
\begin{enumerate}
\item  If $c \geq 0$, $f \geq 0$
\begin{align*}
\int cf\;dH = c \int f\;dH;
\end{align*}
\item  If $0\leq f_n \uparrow f$, then
\begin{align*}
\lim_{n \to \infty}\int f_n\;dH = \int f\;dH;
\end{align*}
\item \label{sub} For non-negative functions $f_n$,  
\begin{align*}
\int \sum_{n=1}^\infty f_n\;dH \leq \sum_{n=1}^\infty \int f_n\;dH.
\end{align*}
\end{enumerate}
\end{proposition}

The property \eqref{sub} asserts that under such assumptions one has a sublinear integral.  When restricted to a suitable class of functions, this will allow us to define a norm.  To this end, we next introduce the notion of $H$-quasicontinuity.
\begin{definition}
Given an outer measure $H$, we say that a function $f$ is $H$-quasicontinuous if  for every $\epsilon>0$, there exists an open set $O_\epsilon$ such that $H(O_\epsilon)<\epsilon$ and $f|_{O_\epsilon^c}$ is continuous.
\end{definition}

We are now prepared to introduce a norm on the space of $H$-quasicontinuous functions.  For a signed $H$-quasicontinuous function $f$ we define
\begin{align}\label{choquet_def}
\|f\|_{L^1(Q_0;H)}\vcentcolon= \int_{Q_0} |f|\;dH,
\end{align}
and we let
\begin{align}\label{l1}
L^1(Q_0;H)\vcentcolon= \left\{ f\; H\text{-quasicontinuous } : \|f\|_{L^1(Q_0;H)}<+\infty\right\}.
\end{align}

The fact that this object is a norm is one part of the following theorem on the properties of this space which are derived from the assumptions on $H$.
\begin{theorem}\label{functionspace}
Suppose $H$ is an outer measure which is strongly subadditive and continuous from below, $Q_0 \subset \mathbb{R}^d$ is a cube, and let $L^1(Q_0;H)$ be as defined in \eqref{l1}.  Then
\begin{enumerate}
\item \eqref{choquet_def} is a norm;
\item $L^1(Q_0;H)$ is a Banach space;
\item The space of continuous and bounded functions in $Q_0$, $C_b(Q_0)$, for which \eqref{choquet_def} is finite is dense in $L^1(Q_0;H)$.
\end{enumerate}
\end{theorem}

\begin{proof}
It follows from Proposition \ref{sublinear} above (Proposition 3.2 in \cite{STW}) that under these assumptions the integral is sublinear, so that the rest of the assertions follow from
 \cite[Proposition 2.2 and 2.3]{OP}.
 \end{proof}

\begin{remark}
If, in addition, one has the property that
\begin{align*}
\sum_{j=0}^\infty H\left(E\cap \{ j\leq |x| < j+1\}\right) \leq C  H(E)
\end{align*}
for all $E \subset \mathbb{R}^d$ and some $C>0$ independent of $E$, then one can replace $C_b(Q_0)$ with $C_c(Q_0)$ in the above density assertion, see \cite[Proposition 2.2]{OP}.  Moreover, one observes from the proof that the sets
\begin{align*}
\{ j\leq |x| < j+1\}
\end{align*}
can be replaced by more general tessellations of $\mathbb{R}^d$, see also \cite[(e) on p.~14]{AdamsChoquet1}
\end{remark}

With these results collected, we next introduce the outer measures of relevance for our consideration.  First, we again recall the definition of the spherical Hausdorff content
\begin{align*}
\mathcal{H}^{\beta}_{\infty}(E):= \inf \left\{\sum_{i=1}^\infty \omega_{\beta} r_i^\beta : E \subset \bigcup_{i=1}^\infty B(x_i,r_i) \right\},
\end{align*}
where $\omega_\beta\vcentcolon= \pi^{\beta/2}/\Gamma(\beta/2+1)$ is a normalization constant.  Next, for any fixed cube $Q \subset \mathbb{R}^d$ we recall the definition of the dyadic Hausdorff content subordinate to $Q$:
\begin{align*}
\mathcal{H}^{\beta,Q}_{\infty}(E):= \inf \left\{\sum_{i=1}^\infty l(Q_i)^\beta : E \subset \bigcup_{i=1}^\infty Q_i, Q_i \in \mathcal{D}(Q) \right\},
\end{align*}
Here $\mathcal{D}(Q)$ denotes the set of all dyadic cubes generated by the cube $Q$, where we follow the convention of \cite{STW,ST} with cubes taken as half-open.  For example, in the case $Q= [0,1)^d$, this is the usual dyadic lattice and in this case $\mathcal{H}^{\beta,Q}_{\infty}=\tilde{\mathcal{H}}^{\beta}_{\infty}$, the usual dyadic Hausdorff content.

While the spherical Hausdorff content and dyadic Hausdorff content adapted to a cube are both outer measures, the spherical Hausdorff content may fail to be strongly subadditive.  This is not the case for the dyadic content, which satisfies the necessary properties to generate a well-behaved Choquet integral.  In particular, a suitable modification of \cite[Proposition 3.5 and 3.6]{STW} to the dyadic lattice generated by $Q$ yields
\begin{proposition}
$\mathcal{H}^{\beta,Q}_{\infty}$ is a strongly subadditive and continuous from below.
\end{proposition}

Therefore as a consequence of Theorem \ref{functionspace}, the function space of $\mathcal{H}^{\beta,Q}_{\infty}$-quasicontinuous functions for which norm \ref{l1} is finite is complete and enjoys approximation by bounded continuous functions.  The motivation for introducing the spherical content here is that it is a useful tool to move between the various dyadic contents, as one has their equivalence as set functions.  The following result asserting this is proved in \cite[Proposition 2.3]{YangYuan} in the usual dyadic case, though the argument extends to a general cube by rotation and dilation.
\begin{proposition}
There exist a universal constant $C_\beta>0$ such that for every $Q \subset \mathbb{R}^d$
\begin{align*}
 \frac{1}{C_\beta} \mathcal{H}^{\beta,Q}_{\infty}(E) \leq  \mathcal{H}^{\beta}_{\infty}(E) \leq C_\beta \mathcal{H}^{\beta,Q}_{\infty}(E)
\end{align*}
for all $E \subset \mathbb{R}^d$.
\end{proposition}

The preceding proposition shows that all of the dyadic Hausdorff contents are equivalent with a constant independent of the cube, so that if one works in a fixed cube $Q_0$,  the functional spaces $L^1(Q_0;\mathcal{H}^{\beta,Q}_\infty)$ are all equivalent (though with different norms).  Therefore if we introduce
\begin{align}\label{norm-prime}
\|f\|_{L^1_*(Q_0;\mathcal{H}^{\beta,Q_0}_\infty)}\vcentcolon= \int_{Q_0} |f|\;d\mathcal{H}^{\beta,Q_0}_\infty,
\end{align}
we have finally collected all of relevant references of several facts claimed in the introduction - that \eqref{norm-prime} is a norm which is equivalent to the quasi-norm \eqref{norm}, that $L^1(Q_0;\mathcal{H}^{\beta,Q_0}_\infty)$ is a Banach space, and that the space of functions in $C_b(Q_0)$ for which \eqref{norm} is finite is dense in $L^1(Q_0;\mathcal{H}^{\beta,Q_0}_\infty)$.  The last of the claims to address is the duality asserted in the introduction, that the dual of $L^1(Q_0;\mathcal{H}^{\beta,Q_0}_\infty)$ is the Morrey space
\begin{align*}
\mathcal{M}^{\beta}(Q_0)\vcentcolon= \left\{ \nu \in M_{loc}(\mathbb{R}^d): \|\nu\|_{\mathcal{M}^\beta} \vcentcolon= \sup_{Q \in \mathcal{D}(Q_0)} \frac{|\nu|(Q)}{\;l(Q)^\beta}<\infty \right\}.
\end{align*}
This is a result of D.R. Adams from \cite[Proposition 1 on p.~118]{AdamsChoquet} (see also the detailed argument presented by H. Saito and H. Tanaka in \cite{ST}).  From this duality the formula \eqref{HB} follows from the Hahn-Banach theorem.  This implies an important consequence we appealed to in the introduction, that for functions $f \in L^1(Q_0;\mathcal{H}^{\beta,Q_0}_\infty)$ one has an upper bound for the integral of $|f|$ with respect to any $\beta$-dimensional growth measure $\nu$:  There exists a constant $C'_\beta>0$ such that
\begin{align*}
\int_{Q_0} |f| \;d\nu  \leq C'_\beta \int_{Q_0} |f|\;d\mathcal{H}^{\beta,Q_0}_\infty
\end{align*}
for every $\nu$ with $\|\nu\|_{\mathcal{M}^\beta} \leq 1$.

Another important consequence of the equivalence of the spherical Hausdorff content with any of the dyadic Hausdorff contents  is that one can replace the semi-norm on $BMO^\beta(Q_0)$ with a quantity which simplifies the proofs, namely
\begin{align}\label{dyadic_norm}
\|u\|_{BMO^{\beta}_*(Q_0)}\vcentcolon=   \sup_{Q \subset Q_0} \inf_{c \in \mathbb{R}} \frac{1}{l(Q)^\beta}  \int_{Q}  |u-c| \;d\mathcal{H}^{\beta,Q}_\infty.
\end{align}
Indeed, we utilize this semi-norm in the sequel in the argument of Lemma \ref{aux_lemma}, which is slightly more general than Theorem \ref{jn_content} in that it can be applied to infinite cubes.  This is in the spirit of the program instituted in John and Nirenberg \cite{JN}, analogous to their Lemma 1'.   The semi-norm \eqref{dyadic_norm} also has the advantage that some of the quantities become more transparent, e.g. that one can take
\begin{align*}
c_Q= \argmin_{c \in \mathbb{R}} \frac{1}{l(Q)^\beta}  \int_{Q}  |u-c| \;d\mathcal{H}^{\beta,Q}_\infty,
\end{align*}
and that this infimum is attained.  Our choice to give definitions and the theorems stated in the introduction in terms of the spherical Hausdorff content stems from our impression that this object is more common in the literature and therefore more accessible to the reader.


Finally, we record a small modification of a lemma from \cite{OV} which is attributed there to Mel'nikov \cite{Mel}.  In particular, an inspection of the proof in \cite{OV}  yields the following statement.
\begin{proposition}\label{ov}
Suppose $\{Q_j\}$ is a family of non-overlapping dyadic cubes subordinate to some cube lattice $\mathcal{D}(Q)$.  There exists a subfamily $\{Q_{j_k}\}$ and a family of non-overlapping ancestors $\tilde{Q}_k$ such that
\begin{enumerate}
\item \begin{align*}
\bigcup_{j} Q_j \subset \bigcup_{k} Q_{j_k} \cup \bigcup_{k} \tilde{Q}_k
\end{align*}
\item
\begin{align*}
\sum_{Q_{j_k} \subset Q}l(Q_{j_k})^\beta \leq 2l(Q)^\beta, \text{ for each dyadic cube } Q.
\end{align*}
\item For each $\tilde{Q}_k$, 
\begin{align*}
l(\tilde{Q}_k)^\beta \leq \sum_{Q_{j_i} \subset \tilde{Q}_k}l(Q_{j_i})^\beta.
\end{align*}
\end{enumerate}
\end{proposition}
The condition (2) is sometimes called a packing condition, an important consequence of which is the existence of a constant $C'>0$ such that 
\begin{align*}
\sum_{k}\int_{Q_{j_k}}  f \;d\mathcal{H}^{\beta,Q}_{\infty} \leq C' \int_{\cup_{k} Q_{j_k}} f\;d\mathcal{H}^{\beta,Q}_{\infty} .
\end{align*}
for every non-negative $f \in L^1(\cup_{k} Q_{j_k};\mathcal{H}^{\beta,Q}_{\infty})$.

\section{Capacitary Differentiation Theorems}\label{content-jn}

The purpose of this section is to establish for the dyadic Hausdorff content adapted to an arbitrary cube some analogues of classical results known for the Lebesgue measure.


%
%

The first result of this section is an analogue of D. R. Adams' result \cite{AdamsChoquet} on the boundedness of the Hardy-Littlewood maximal function on the space 
\begin{align*}
L^1(\mathbb{R}^d;\mathcal{H}^\beta_\infty) 
\end{align*}
for the dyadic $\beta$-Hausdorff maximal function associated to $\mathcal{H}^{\beta,Q}_{\infty}$:
\begin{align*}
\mathcal{M}^{\beta,Q}_\infty f(x) \vcentcolon= \sup_{Q' \in\mathcal{D}(Q)} \chi_{Q'}(x) \frac{1}{l(Q')^\beta} \int_{Q'} |f| \;\;d\mathcal{H}^{\beta,Q}_{\infty}.
\end{align*}
Note that the equivalence of all of the Hausdorff contents implies that
\begin{align*}
L^1(\mathbb{R}^d;\mathcal{H}^\beta_\infty) = L^1(\mathbb{R}^d;\mathcal{H}^{\beta,Q}_{\infty})
\end{align*}

In particular we prove
\begin{theorem}\label{max}
Let $C'>0$ be the constant such that
\begin{align*}
\sum_{k}\int_{Q_{k}}  f \;d\mathcal{H}^{\beta,Q}_{\infty} \leq C' \int_{\cup_{k} Q_{k}} f\;d\mathcal{H}^{\beta,Q}_{\infty} .
\end{align*}
for every family of non-overlapping dyadic cubes $\{Q_{k}\}$ subordinate to $Q$ which satisfy
\begin{align*}
\sum_{Q_{k} \subset Q'}l(Q_{k})^\beta \leq 2l(Q')^\beta, \text{ for each dyadic cube } Q' \in \mathcal{D}(Q).
\end{align*}
Then
\begin{align*}
\mathcal{H}^{\beta,Q}_{\infty}\left(\left\{ \mathcal{M}^{\beta,Q}_\infty f>t\right\}\right) \leq \frac{C'}{t} \int_{\mathbb{R}^d} f\;d\mathcal{H}^{\beta,Q}_{\infty}
\end{align*}
for all $f \in L^1(\mathbb{R}^d;\mathcal{H}^{\beta,Q}_{\infty})$. 
\end{theorem}

\begin{proof}
Without loss of generality we suppose $f$ is non-negative.  By the definition of the dyadic maximal function we have 
\begin{align*}
\{\mathcal{M}^{\beta,Q}_\infty f >t \} = \bigcup_{i=1}^\infty Q_i
\end{align*}
where each of the $Q_i$ satisfies
\begin{align}\label{maximal_cube_t}
   l(Q_i)^\beta < \frac{1}{t} \int_{Q_i} f \; d\mathcal{H}^{\beta,Q}_{\infty}.
\end{align}
Let $\{Q_j \}$ be the maximal collection of these dyadic cubes, each $Q_j$ of which satisfies \eqref{maximal_cube_t} and let $\{Q_{j_k}\}$ and $\{\tilde{Q}_k\}$ be the families of cubes obtained by an application of Proposition \ref{ov} to this family.  By only selecting necessary cubes,  $Q_{j_k} \not \subseteq \tilde{Q}_m$ for some $Q_m$ in the original family with $\tilde{Q}_m$ selected as a covering cube, we have 
\begin{align*}
    \mathcal{H}^{\beta,Q}_\infty (\{\mathcal{M}^{\beta,Q}_\infty f >t \}) &\leq \sum_{k, Q_{j_k} \not \subseteq \tilde{Q}_m} \mathcal{H}^{\beta,Q}_\infty (Q_{j_k}) +\sum_k \mathcal{H}^{\beta,Q}_\infty (\tilde{Q}_k) \\
    &= \sum_{k, Q_{j_k} \not \subseteq \tilde{Q}_m} l(Q_{j_k})^\beta +\sum_k l(\tilde{Q}_k)^\beta \\
    &\leq \sum_{k} l(Q_{j_k})^\beta.
\end{align*}
However, since $Q_{j_k}$ satisfies \eqref{maximal_cube_t} we find that
\begin{align*}
       \mathcal{H}^{\beta,Q}_\infty (\{\mathcal{M}^{\beta,Q}_\infty f >t \})  \leq  \sum_{k} \frac{1}{t} \int_{Q_{j_k}} f \; d\mathcal{H}^{\beta,Q}_{\infty}.
\end{align*}
Finally, the packing condition satisfied by $Q_{j_k}$ implies
\begin{align*}
   \sum_{k} \frac{1}{t} \int_{Q_{j_k}} f \; d\mathcal{H}^{\beta,Q}_{\infty} \leq C' \frac{1}{t} \int_{\cup Q_{j_k}} f \; d\mathcal{H}^{\beta,Q}_{\infty},
\end{align*}
which gives the desired conclusion.
\end{proof}

We next establish a precise estimate for upper density of the dyadic Hausdorff content.  Our argument follows that of Evans and Gariepy \cite[p.~74]{EvansGariepy}, where we observe that the use of the dyadic content yields a slightly better result than the upper and lower bounds one has for the upper density of the spherical content.
\begin{lemma}\label{density}
Let $E \subset \mathbb{R}^n$. Then
\begin{align*}
    \limsup_{Q' \to x} \frac{\mathcal{H}^{\beta,Q}_\infty (E \cap Q')}{\ell (Q')^\beta} =1
\end{align*}
for $\mathcal{H}^\beta$ -a.e. $x \in E,$ where the limit is taken over all dyadic cubes $Q' \in \mathcal{D}(Q)$ containing $x$ as $l(Q')$ tends to 0. 
\end{lemma}

\begin{proof}

As density is a local result, by taking $E\cap Q$ for a sufficiently large cube $Q$, we may assume $E$ is bounded and therefore $\mathcal{H}^{\beta,Q}_\infty (E)<+\infty$.

For every $\delta > 0$ and $0< \tau <1$, we define 
\begin{align*}
    E(\delta, \tau ) : = \left\{ x \in E : \  \frac{\mathcal{H}^{\beta,Q}_\infty (E \cap Q')}{l(Q')^\beta } \leq \tau \text{ for every } Q' \in \mathcal{D}(Q) \text{ such that } x\in Q',\; l(Q') \leq \delta \right\}.
\end{align*}
We claim that $\mathcal{H}^\beta(E(\delta, \tau)) =0$ for every $\delta > 0$ and $0< \tau <1$.  To this end, we introduce
\begin{align*}
\mathcal{H}^{\beta,Q}_{\delta}(E)\vcentcolon= \inf \left\{\sum_{i=1}^\infty l(Q_i)^\beta : E \subset \bigcup_{i=1}^\infty Q_i, Q_i \in \mathcal{D}(Q), l(Q_i)\leq \delta \right\},
\end{align*} 
and observe that $\mathcal{H}^{\beta,Q}_{\delta}$ is an outer measure.

Thus for every collection of nonoverlapping dyadic cubes $\{ Q_i \}_{i=1}^\infty$ such that $E(\delta,\tau) \subset \cup^\infty_{i = 1} Q_i$, $Q_i \cap E(\delta, \tau) \neq \emptyset$ and $l(Q_i) \leq \delta$, we have by countable subadditivity, monotonicity, and the fact that $Q_i \cap E(\delta, \tau) \neq \emptyset$ with $l(Q_i) \leq \delta$
\begin{align*}
\mathcal{H}^{\beta,Q}_\delta(E(\delta, \tau)) &\leq \sum^\infty_{i =1 }\mathcal{H}^{\beta,Q}_\delta(E(\delta, \tau) \cap Q_i)\\
& \leq  \sum^\infty_{i =1 }\mathcal{H}^{\beta,Q}_\delta(E \cap Q_i)\\
& =  \sum^\infty_{i =1 }\mathcal{H}^{\beta,Q}_\infty(E \cap Q_i)\\
& \leq \tau  \sum^\infty_{i =1 } l(Q_i)^\beta.
\end{align*}
The fact $E(\delta,\tau) \subset E$ and $\mathcal{H}^{\beta,Q}_\infty (E)<+\infty$ implies $\mathcal{H}^{\beta,Q}_\infty (E(\delta,\tau))<+\infty$, and therefore $\mathcal{H}^{\beta,Q}_\delta (E(\delta,\tau))<+\infty$.  Thus, taking the infimum over all such families of cubes in the preceding inequality we deduce that $\mathcal{H}^{\beta,Q}_\delta(E(\delta, \tau)) \leq \tau \mathcal{H}^{\beta,Q}_\delta(E(\delta, \tau))$. Since $0< \tau<1$, we have $\mathcal{H}^{\beta,Q}_\delta(E(\delta, \tau)) =0$ and therefore $\mathcal{H}^\beta (E(\delta, \tau)) =0$ (which follows from \cite[Lemma 1 on p.~64]{EvansGariepy} and the equivalence of the dyadic and spherical contents, for example), which proves the claim.

Now suppose that $x \in E $ and 
\begin{align*}
\limsup\limits_{Q' \to x}  \frac{\mathcal{H}^{\beta,Q}_\infty(E \cap Q')}{l(Q')^\beta } < 1.
\end{align*}
Then using the fact that the inequality is strict and the definition of the limit supremum, we can find $\delta=\delta_x \in (0,1)$ such that 
\begin{align*}
\frac{\mathcal{H}^{\beta,Q}_\infty(E\cap Q')}{l(Q')^\beta} \leq 1-\delta
\end{align*}
for every dyadic cube $Q'$ satisfying $x\in Q'$ and $l(Q') \leq \delta,$ which means that $ x\in E(\delta_x, 1-\delta_x)$. Using the fact that 
\begin{align*}
    E(\delta_1 ,1- \delta_1) \subset E(\delta_2, 1-\delta_2) 
\end{align*}
for $0<\delta_2 \leq \delta_1 <1$, this means that 
\begin{align*}
x \in E\left(\frac{1}{k}, 1- \frac{1}{k}\right)
\end{align*}
for any $k \in \mathbb{N}$ such that $1/k<\delta_x$.  In particular, we have the set inclusion
\begin{align*}
    \left\{ x\in E : \ \limsup_{Q' \to x } \frac{\mathcal{H}^{\beta,Q}_\infty(E \cap Q')}{l(Q')^\beta } < 1 \right\} \subset \bigcup\limits^\infty_{k =1} E\left(\frac{1}{k}, 1- \frac{1}{k}\right).
\end{align*}
As the sets on the right hand side have $\mathcal{H}^\beta$ measure zero, we deduce
\begin{align*}
    \mathcal{H}^\beta \left(\left\{ x\in E : \ \limsup_{Q \to x } \frac{\mathcal{H}^\beta_\infty(Q \cap E)}{\ell(Q)^\beta } < 1 \right\}\right) =0,
\end{align*}
and the proof is complete.
\end{proof}

Next, we establish a weak Lebesgue differentiation theorem for the Hausdorff content:
\begin{theorem}\label{differentialhausdorfflimsup}
Let $f \in L^1(\mathbb{R}^d;\mathcal{H}^{\beta,Q}_\infty)$ be non-negative. Then for $\mathcal{H}^\beta$ -a.e. $x\in \mathbb{R}^d$, we have 
\begin{align}
    \limsup_{Q' \to x} \frac{1}{l(Q')^\beta} \int_{Q'} f\; d\mathcal{H}^{\beta,Q}_\infty = f(x).
\end{align}
\end{theorem}
\begin{proof}

We first prove that for $\mathcal{H}^\beta$ -a.e. $x\in \mathbb{R}^d$, we have
\begin{align}\label{Lebesge_smaller}
     \limsup_{Q' \to x} \frac{1}{l(Q')^\beta} \int_{Q'} f\;  d\mathcal{H}^{\beta,Q}_\infty \leq  f(x).
\end{align}
For $\epsilon>0$, we define 
\begin{align*}
    A_\epsilon\vcentcolon= \left\{ x\in \mathbb{R}^n : \ \epsilon <  \limsup_{Q' \to x} \frac{1}{l(Q')^\beta} \int_{Q'} f\;  d\mathcal{H}^{\beta,Q}_\infty - f(x) \right\}.
\end{align*}
Note that inequality (\ref{Lebesge_smaller}) will be verified once we show that $\mathcal{H}^{\beta,Q}_\infty(A_\epsilon) = 0$ for every $\epsilon>0.$
Since $f \in L^1(\mathbb{R}^d; \mathcal{H}^{\beta,Q}_\infty)$, there exists a sequence of functions $\{f_k\} \subset C_b(\mathbb{R}^d) \cap L^1(\mathbb{R}^d;\mathcal{H}^{\beta,Q}_\infty)$ such that 
\begin{align*}
    \int_{\mathbb{R}^d} |f - f_k| \;d\mathcal{H}^{\beta,Q}_\infty \leq \frac{1}{k}.
\end{align*}
In particular, we have by subadditivity of the integral
\begin{align*}
    \frac{1}{l(Q')^\beta} \int_{Q'} f\;  d\mathcal{H}^{\beta,Q}_\infty - f(x) & \leq \frac{1}{l(Q')^\beta} \int_{Q'} |f(\cdot) - f(x)| \; d \mathcal{H}^{\beta,Q}_\infty\\
    &\leq \frac{1}{l(Q')^\beta} \int_{Q'} |f(\cdot) - f_k(\cdot)| + |f_k(\cdot) - f_k(x)|+ |f_k(x) - f(x)| \; d\mathcal{H}^{\beta,Q}_\infty\\
    &\leq \frac{1}{l(Q')^\beta}\int_{Q'}  |f(\cdot) - f_k(\cdot)|  \; d\mathcal{H}^{\beta,Q}_\infty \\
    &+ \frac{1}{l(Q')^\beta}\int_{Q'} |f_k(\cdot) - f_k(x)| \; d\mathcal{H}^{\beta,Q}_\infty\\
    &+ |f_k(x) - f(x)|.
    \end{align*}
Since $f_k$ is continuous for each $k$, we have 
\begin{align*}
    \limsup_{Q' \to x}\frac{1}{l(Q')^\beta}\int_{Q'} |f_k(\cdot) - f_k(x)| \; d\mathcal{H}^{\beta,Q}_\infty = 0.
\end{align*}
It follows that for every $\epsilon> 0$, we have the set inclusion
\begin{align*}
    A_\epsilon \subset \left \{ x\in \mathbb{R}^n : \  \mathcal{M}_\infty^\beta |f - f_k|(x) > \frac{\epsilon}{2}   \right\} \cup \left\{ x\in \mathbb{R}^n : \  |f_k(x) - f(x)| > \frac{\epsilon}{2}  \right\}: = A^{1,k}_\epsilon \cup A^{2,k}_\epsilon.
\end{align*}
and therefore
\begin{align*}
  \mathcal{H}^{\beta,Q}_\infty(A_\epsilon) \leq  \mathcal{H}^{\beta,Q}_\infty(A^{1,k}_\epsilon) + \mathcal{H}^{\beta,Q}_\infty(A^{2,k}_\epsilon).
\end{align*}
Finally, an application of Theorem \ref{max} to $A^{1,k}_\epsilon$ and Chebychev's inequality to $A^{2,k}_\epsilon$, along with the density of $C_0(\mathbb{R}^d)$, yields that 
\begin{align*}
\limsup_{k \to \infty} \mathcal{H}^{\beta,Q}_\infty(A^{1,k}_\epsilon) &=0,\\
\limsup_{k \to \infty} \mathcal{H}^{\beta,Q}_\infty(A^{2,k}_\epsilon)&=0.
\end{align*}
As the Hausdorff content and Hausdorff measure have the same null sets \cite[Lemma 1 on p.~64]{EvansGariepy}, this completes the proof of \eqref{Lebesge_smaller}.

We next show that for $\mathcal{H}^\beta$ -a.e. $x\in \mathbb{R}^n$, we have
\begin{align}\label{differentialgeq}
     \limsup_{Q' \to x} \frac{1}{l(Q')^\beta} \int_{Q'} f\;  d\mathcal{H}^{\beta,Q}_\infty \geq f(x).
\end{align}
Let $L_t: = \{ x \in \mathbb{R}^d : \ f(x )> t\}$ for every $t>0$. Then for every $x \in L_t$, we have
\begin{align*}
    \limsup_{Q' \to x} \frac{1}{l(Q')^\beta} \int_{Q'} f\;  d\mathcal{H}^{\beta,Q}_\infty
    & \geq  \limsup_{Q' \to x} \frac{1}{l(Q')^\beta} \int_{Q' \cap L_t} t\;  d\mathcal{H}^{\beta,Q}_\infty\\
    & = \limsup_{Q' \to x}\frac{\mathcal{H}^{\beta,Q}_\infty(Q' \cap L_t)}{l(Q')^\beta}  \ t.
\end{align*}
An application of Lemma \ref{density} yields
\begin{align*}
     \limsup_{Q' \to x} \frac{1}{l(Q')^\beta} \int_{Q'} f\;  d\mathcal{H}^{\beta,Q}_\infty \geq t
\end{align*}
for $\mathcal{H}^\beta$ -a.e. $x \in L_t$. That is, for every $t>0$, we have
\begin{align*}
   \mathcal{H}^\beta\left( \left\{x \in \mathbb{R}^d : \      \limsup_{Q' \to x} \frac{1}{l(Q')^\beta} \int_{Q'} f\;  d\mathcal{H}^{\beta,Q}_\infty< t <f(x) \right\}      \right) = 0.
\end{align*}
In particular, we observe that
\begin{align*}
   & \left\{x \in \mathbb{R}^d : \      \limsup_{Q' \to x} \frac{1}{l(Q')^\beta} \int_{Q'} f\;  d\mathcal{H}^{\beta,Q}_\infty <f(x) \right\} \\
    \subset &\bigcup_{t \in \mathbb{Q}^{+}} \left\{x \in \mathbb{R}^d : \      \limsup_{Q' \to x} \frac{1}{l(Q')^\beta} \int_{Q'} f\;  d\mathcal{H}^{\beta,Q}_\infty< t <f(x) \right\}, 
\end{align*}
where $\mathbb{Q^+}$ denotes the set of all positive rational numbers. Using the subadditivity property of $\mathcal{H}^\beta$, we conclude that 
\begin{align*}
    \mathcal{H}^\beta\left( \left\{x \in \mathbb{R}^d : \ \limsup_{Q' \to x} \frac{1}{l(Q')^\beta} \int_{Q'} f \; d\mathcal{H}^{\beta,Q}_\infty <f(x) \right\} \right) = 0,
\end{align*}
which implies the claim \eqref{differentialgeq}.
\end{proof}

As we show, this one-sided Lebesgue differentiation theorem is sufficient to obtain a Calder\'on-Zygmund decomposition of these spaces:
\begin{theorem}\label{cz}
Let $Q$ be a cube in $\mathbb{R}^d$ and $f \in L^1(Q,\mathcal{H}^{\beta,Q}_\infty)$. Then for every $\lambda>0$ such that 
\begin{align*}
   \lambda \geq \frac{1}{l(Q)^\beta} \int_Q |f| \; d\mathcal{H}^{\beta,Q}_\infty,
\end{align*}
there exists a countable collection of non-overlapping dyadic cubes $\{ Q_k\}\subset Q$ subordinate to $Q$ such that 
\begin{enumerate}
    \item $\lambda< \frac{1}{l(Q_k)^\beta} \int_{Q_k} |f| \; \mathcal{H}^{\beta,Q}_\infty \leq 2^\beta \lambda$,
    \item $|f(x)| \leq \lambda$ for $\mathcal{H}^\beta$ -a.e. $x\in Q \setminus \bigcup_k Q_k$,
\end{enumerate}

\end{theorem}
\begin{proof}
Let $\{ Q_i \}$ be the collection of all the dyadic cubes contained in $Q$ which satisfies
\begin{align*}
    \lambda< \frac{1}{l(Q_i)^\beta} \int_{Q_i} |f| \; d\mathcal{H}^{\beta,Q}_\infty
\end{align*}
and $\{ Q_k \}$ be the maximal subcollection of $\{ Q_i \}$.  Then for every $Q_k$, its parent cube $Q_k'$ satisfies
\begin{align*}
    \lambda \geq \frac{1}{l(Q_k')^\beta} \int_{Q_k'} |f| \; d\mathcal{H}^{\beta,Q}_\infty \geq \frac{1}{2^\beta l(Q_k)^\beta} \int_{Q_k} |f| \; d\mathcal{H}^{\beta,Q}_\infty
\end{align*}
and thus
\begin{align*}
    \frac{1}{l(Q_k)^\beta} \int_{Q_k} |f| \; d\mathcal{H}^{\beta,Q}_\infty \leq 2^\beta \lambda,
\end{align*}
which is the first claimed property.  Toward the second claimed property, for $x\in Q \setminus \bigcup_k Q_k $, for every dyadic cube $Q'$ contained in $Q$ with $x \in Q'$, the fact that $Q'$ is not in any of the selected cubes means 
\begin{align*}
  \frac{1}{l(Q')^\beta} \int_{Q'} |f| \; d\mathcal{H}^{\beta,Q}_\infty\leq \lambda.
\end{align*}
In particular, by Theorem \ref{differentialhausdorfflimsup}, we obtain that 
$|f(x)| \leq \lambda$ for $\mathcal{H}^\beta$ -a.e. $x\in Q \setminus \bigcup_k Q_k$.
\end{proof}

\section{Main Result and Several Consequences}\label{mainresult}
The goal of this section is to prove Theorem \ref{jn_content} and then use it to deduce the several corollaries from the introduction.  We begin with an exponential decay lemma which is a slight variant of a result from the paper of John and Nirenberg \cite{JN}.

\begin{lemma}\label{exponentialestimate}
  Suppose that for some $C',c_\beta>1$, $F: (0,\infty) \to [0,\infty)$ satisfies
  \begin{align} \label{exp_decay}
      F(t) \leq C' \frac{F(t-c_\beta s)}{s}
  \end{align}
  for all $c_\beta^{-1} t \geq s \geq 1$.  If
  \begin{align*}
  F(t) \leq \frac{1}{t}
  \end{align*}
  for every $t>0$, then there exist constants $c,C>0$ such that
  \begin{align*}
      F(t) \leq  C \exp\{ -c t\}
  \end{align*}
  for every $t\geq c_\beta$.  
  \end{lemma}
\begin{proof}
For $t \geq c_\beta$, let $m \in \mathbb{N}\cup \{0\}$ be such that
\begin{align*}
c_\beta + C'c_\beta e m \leq t <  c_\beta + C'c_\beta e (m+1).
\end{align*}
Iteration of the choice $s=C'e$ in the inequality \eqref{exp_decay} a total of $m$ times yields
\begin{align*}
    F(t) &\leq C' \frac{F(t -C'c_\beta e )}{C'e} = \exp(-1) F(t -C'c_\beta e ) \leq \ldots  \leq \exp(-m) F(t -C'c_\beta e m ) \\
    &\leq  \frac{\exp(-m)}{c_\beta},
 \end{align*}
while the upper bound $ t <  c_\beta + C'c_\beta e (m+1)$ implies
\begin{align*}
-m \leq \frac{-t+c_\beta}{C'c_\beta e}+1.
\end{align*}
In particular, using the monotonicity properties of the exponential and $c_\beta>1$ we find
\begin{align*}
    F(t) &\leq  \frac{1}{c_\beta}\exp\left(\frac{-t+c_\beta}{C'c_\beta e}+1\right)  \\
    &\leq  \exp\left(\frac{1}{C'e}+1\right) \exp\left(\frac{-t}{C'c_\beta e}\right)
\end{align*}
for $t \geq c_\beta$, so that the choice
\begin{align*}
C&\vcentcolon=\exp\left(\frac{1}{C'e}+1\right) \\
c&\vcentcolon=\frac{1}{C'c_\beta e}
\end{align*}
gives the desired conclusion for $t \geq c_\beta$.  
\end{proof}

Theorem \ref{jn_content} is an almost immediate consequence of the following analogue of \cite[Lemma 1']{JN}:
\begin{lemma}\label{aux_lemma}
There exist constants $c,C>0$ such that
\begin{align}\label{jninequality_content}
\mathcal{H}^{\beta,Q}_\infty\left(\{x\in Q:|u(x)-c_Q|>t\}\right) \leq \frac{C}{\|u\|_{BMO^\beta_*(Q_0)}}  \int_{Q}  |u-c_Q|\; d\mathcal{H}^{\beta,Q}_\infty  \exp\left(\frac{-ct}{\|u\|_{BMO^\beta_*(Q_0)}}\right)
\end{align}
for all $u \in BMO^\beta(Q_0)$, finite subcubes $Q \subset Q_0$ parallel to $Q_0$, and $t\geq c_\beta \|u\|_{BMO^\beta_*(Q_0)}$, where
\begin{align*}
c_Q= \argmin_{c \in \mathbb{R}} \frac{1}{l(Q)^\beta}  \int_{Q}  |u-c| \;d\mathcal{H}^{\beta,Q}_\infty
\end{align*}
and $c_\beta=1+2^\beta$.
\end{lemma}

\begin{proof}
We assume without loss of generality that $c_Q=0$ and $\|u\|_{BMO^\beta_*(Q_0)}=1$, by replacing $u$ with $\frac{u-c_Q}{\|u\|_{BMO^\beta_*(Q_0)}}$.  Let $F(t)$ be the smallest function for which
\begin{align}\label{jninequality_content_F}
\mathcal{H}^{\beta,Q}_\infty\left(\{x\in Q:|u(x)|>t\}\right) \leq F(t)  \int_{Q}  |u|\; d\mathcal{H}^{\beta,Q}_\infty
\end{align}
for all $t> 0$, $u \in BMO^\beta(Q_0)$ with $\|u\|_{BMO^\beta_*(Q_0)}=1$ and finite subcubes $Q \subset Q_0$ parallel to $Q_0$.  As Chebychev's inequality implies
\begin{align*}
F(t) \leq \frac{1}{t},
\end{align*}
it only remains to show that
\begin{align*}
F(t) \leq C' \frac{F(t-c_\beta s)}{s}
\end{align*}
for all $c_\beta^{-1}t \geq s \geq 1$, where $c_\beta=1+2^\beta$, from which the desired inequality follows by an application of Lemma \ref{exponentialestimate}.  

The reduction that $c_Q=0$ and $\|u\|_{BMO^\beta_*(Q_0)}=1$ implies that
\begin{align} \label{upperbound}
\frac{1}{l(Q)^\beta} \int_{Q}  |u| \;d\mathcal{H}^{\beta,Q}_{\infty} \leq 1.
\end{align}
In particular, if 
\begin{align*}
c_\beta^{-1}t \geq s \geq 1,
\end{align*}
$\lambda = s$ is an admissible height for an application of Theorem \ref{cz} to $u$ to obtain a countable collection of non-overlapping dyadic cubes $\{ Q_j\} \subset \mathcal{D}(Q)$  such that 
\begin{enumerate}
    \item $ s< \frac{1}{l(Q_j)^\beta} \int_{Q_j} |u| \; d\mathcal{H}^{\beta,Q}_\infty \leq 2^\beta  s$,
    \item $|u(x)|\leq  s$ for $\mathcal{H}^\beta$ -a.e. $x\in Q \setminus \cup_j Q_j$.
\end{enumerate}
An application of Proposition \ref{ov} to this family yields a subfamily $\{ Q_{j_k}\}$ and non-overlapping ancestors $\{ \tilde{Q}_k\}$ which satisfy
\begin{enumerate}
\item \begin{align*}
\bigcup_{j} Q_j \subset \bigcup_{k} Q_{j_k} \cup \bigcup_{k} \tilde{Q}_k
\end{align*}
\item
\begin{align*}
\sum_{Q_{j_k} \subset Q}l(Q_{j_k})^\beta \leq 2l(Q)^\beta, \text{ for each dyadic cube } Q.
\end{align*}
\item For each $\tilde{Q}_k$, \begin{align*}
l(\tilde{Q}_k)^\beta \leq \sum_{Q_{j_i} \subset \tilde{Q}_k}l(Q_{j_i})^\beta.
\end{align*}
\end{enumerate}
As before, we assume we avoid redundancy in this covering so that if $Q_{j_k} \subset \tilde{Q}_m$ for some $k,m$ we do not utilize the cube $Q_{j_k}$.  The fact that
\begin{align*}
t \geq c_\beta s > s
\end{align*}
implies
\begin{align*}
\{x\in Q:|u(x)|>t\} \subset \bigcup_{j} Q_j \cup N \subset \bigcup_{k, Q_{j_k} \not \subset \tilde{Q}_{m}} Q_{j_k} \cup \bigcup_{k} \tilde{Q}_k \cup N,
\end{align*}
where $N$ is a null set with respect to $\mathcal{H}^\beta$ (and therefore also $\mathcal{H}^{\beta,Q}_\infty$).  As the inclusion is preserved if we intersect with $\{x\in Q:|u(x)|>t\}$ we have
\begin{align*}
\{x\in Q:|u(x)|>t\} \subset  \bigcup_{k, Q_{j_k} \not\subset \tilde{Q}_{m}}  \{x\in Q_{j_k}:|u(x)|>t\} \cup \bigcup_{k}  \{x\in \tilde{Q}_k:|u(x)|>t\} \cup N,
\end{align*}
so that by subadditivity of the Hausdorff content we can estimate
\begin{align*}
\mathcal{H}^{\beta,Q}_\infty\left(\{x\in Q:|u(x)|>t\}\right) &\leq \sum_{k, Q_{j_k} \not \subset \tilde{Q}_{m}} \mathcal{H}^{\beta,Q}_\infty\left(\{x\in Q_{j_k}:|u(x)|>t\}\right) \\
&\;\;+ \sum_k \mathcal{H}^{\beta,Q}_\infty\left(\{x\in \tilde{Q}_k :|u(x)|>t\}\right).
\end{align*}
Next, for $Q'=Q_{j_k},\tilde{Q}_k$ we estimate
\begin{align*}
c_{Q'}&= \argmin_{c \in \mathbb{R}} \frac{1}{l(Q')^\beta}  \int_{Q'}  |u-c| \;d\mathcal{H}^{\beta,Q'}_\infty.
\end{align*}
We have, by subadditivity of the dyadic Hausdorff content and using $s\geq 1$,
\begin{align*}
|c_{Q'}| &= \frac{1}{l(Q')^\beta} \int_{Q'} |c_{Q'}|d\mathcal{H}^{\beta,Q'}_\infty \\
&\leq \frac{1}{l(Q')^\beta} \int_{Q'} |u-c_{Q'}|\;d\mathcal{H}^{\beta,Q'}_\infty + \frac{1}{l(Q')^\beta} \int_{Q'} |u|\;d\mathcal{H}^{\beta,Q'}_\infty \\
&\leq 1 + 2^\beta  s \\
&\leq (1+2^{\beta}) s\\
&= c_\beta s.
\end{align*}
In particular, 
\begin{align*}
|u| &\leq |u-c_{Q'}| + |c_{Q'}| \\
&\leq  |u-c_{Q'}| + c_\beta s
\end{align*} 
so that
\begin{align*}
\mathcal{H}^{\beta,Q}_\infty\left(\{x\in Q:|u(x)|>t\}\right) &\leq \sum_{k, Q_{j_k} \not \subset \tilde{Q}_{m}} \mathcal{H}^{\beta,Q}_\infty\left(\{x\in Q_{j_k}:|u(x)-c_{Q_{j_k}}|>t-c_\beta s\}\right) \\
&\;\;+ \sum_k \mathcal{H}^{\beta,Q}_\infty\left(\{x\in \tilde{Q}_{k}:|u(x)-c_{\tilde{Q}_{k}}|>t-c_\beta s\}\right).
\end{align*}
As $\|u-c_{Q'}\|_{BMO^\beta_*(Q_0)} =  \|u\|_{BMO^\beta_*(Q_0)} =1$, using the definition of $F$ we have
\begin{align*}
\mathcal{H}^{\beta,Q}_\infty\left(\{x\in Q:|u(x)|>t\}\right) &\leq \sum_{k, Q_{j_k} \not \subset \tilde{Q}_{m}} F(t-c_\beta s) \int_{Q_{j_k}} |u-c_{Q_{j_k}}|\;d \mathcal{H}^{\beta,Q}_{\infty} \\
&\;\;+ \sum_k F(t-c_\beta s) \int_{\tilde{Q}_{k}} |u-c_{\tilde{Q}_{k}}|\;d \mathcal{H}^{\beta,Q}_{\infty}\\
&\leq \sum_{k, Q_{j_k} \not \subset \tilde{Q}_{m}}  F(t-c_\beta s) l(Q_{j_k})^\beta + \sum_k  F(t-c_\beta s) l(\tilde{Q}_{k})^\beta\\
&\leq  F(t-c_\beta s) \sum_k  l(Q_{j_k})^\beta,
\end{align*}
where we utilize that there is no redundancy in cubes in the collections.  Next, we recall that $\{Q_{j_k}\}$ is a subcollection of cubes from the Calder\'on-Zygmund decomposition and therefore each $Q_{j_k}$ satisfies
\begin{align*}
l(Q_{j_k})^\beta <\frac{1}{ s} \int_{Q_{j_k}} |u| d\mathcal{H}^{\beta,Q}_{\infty}.
\end{align*}
The use of this inequality in the preceding chain of inequalities leads to the estimate
\begin{align*}
\mathcal{H}^{\beta,Q}_{\infty}\left(\{x\in Q:|u(x)|>t\}\right) &\leq F(t-c_\beta s) \sum_k  \frac{1}{ s} \int_{Q_{j_k}} |u| d\mathcal{H}^{\beta,Q}_{\infty} \\
&\leq F(t-c_\beta s) \frac{C'}{ s} \int_{\bigcup_k Q_{j_k}} |u| d\mathcal{H}^{\beta,Q}_{\infty}\\
&\leq F(t-c_\beta s) \frac{C'}{s} \int_{Q} |u| d\mathcal{H}^{\beta,Q}_{\infty}.
\end{align*}
This completes the proof of the claim and therefore the theorem.
\end{proof}

Theorem \ref{jn_content} is an immediate consequence of Lemma \ref{jninequality_content}, when one takes into account the constants computed in Lemma \ref{exponentialestimate}:

\begin{proof}
From the proof of Lemma \ref{jninequality_content} we have
\begin{align*}
\mathcal{H}^{\beta,Q}_{\infty}\left(\{x\in Q:|u(x)-c_{Q}|>t\}\right) &\leq C \exp\left(\frac{-ct}{\|u\|_{BMO^\beta_*(Q_0)}}\right)  \frac{1}{\|u\|_{BMO^\beta(Q_0)}} \int_{Q}  |u-c_{Q}| d\mathcal{H}^{\beta,Q}_{\infty} \\
&\leq C  \exp\left(\frac{-ct}{\|u\|_{BMO^\beta_*(Q_0)}}\right)  l(Q)^\beta
\end{align*}
for all $t \geq c_\beta \|u\|_{BMO^\beta_*(Q_0)}$.  But if $t \in (0,c_\beta \|u\|_{BMO^\beta_*(Q_0)})$, we have
\begin{align*}
\mathcal{H}^{\beta,Q}_{\infty}\left(\{x\in Q:|u(x)-c_{Q}|>t\}\right) &\leq  l(Q)^\beta \\
&\leq l(Q)^\beta \exp\left(\frac{-t}{c_\beta \|u\|_{BMO^\beta_*(Q_0)}}+1\right) \\
&= l(Q)^\beta \exp(1) \exp\left(\frac{-t}{c_\beta \|u\|_{BMO^\beta_*(Q_0)}}\right) \\
&\leq  l(Q)^\beta \exp\left(\frac{1}{C'e}+1\right) \exp\left(\frac{-t}{C'c_\beta e\|u\|_{BMO^\beta_*(Q_0)}}\right)\\
&=  l(Q)^\beta C \exp\left(\frac{-ct}{\|u\|_{BMO^\beta_*(Q_0)}}\right) \\
&\leq  l(Q)^\beta C \exp\left(\frac{-ct}{\|u\|_{BMO^\beta_*(Q_0)}}\right)
\end{align*}
where 
\begin{align*}
C&\vcentcolon=\exp\left(\frac{1}{C'e}+1\right) \\
c&\vcentcolon=\frac{1}{C'c_\beta e}
\end{align*}
are the same constants as above derived in Lemma \ref{exponentialestimate}.  It only remains to replace the dyadic Hausdorff content adapted to $Q$ with the Hausdorff content, which by the equivalence recalled in Section \ref{preliminaries} yields the desired estimate for all $t>0$ with the constants
\begin{align*}
C&\vcentcolon=C_\beta \exp\left(\frac{1}{C'e}+1\right) \\
c&\vcentcolon=\frac{1}{C_\beta} \frac{1}{C'c_\beta e}.
\end{align*}
This completes the proof of Theorem \ref{jn_content}.
\end{proof}

\subsection{Consequences of the John-Nirenberg Inequality}\label{consequences}
We next give the proof of Corollary \ref{exp_integrability}.

\begin{proof}
From Theorem \ref{jn_content} and the fact that $\|u\|_{BMO^\beta(Q_0)}\leq 1$, we have the estimate
\begin{align*}
\mathcal{H}^{\beta}_{\infty} \left(\{x \in Q : |u-c_Q| >t\}\right) &\leq C\exp\left(\frac{-ct}{\|u\|_{BMO^\beta(Q_0)}}\right)l(Q)^\beta \\
 &\leq C\exp\left(-ct\right)l(Q)^\beta
\end{align*}
Thus, for $c'>0$ we find
\begin{align*}
\int_Q \exp\left(c'|u-c_Q|\right) \;d\mathcal{H}^{\beta}_{\infty} &= \int_0^\infty \mathcal{H}^{\beta}_{\infty}\left(\{x \in Q : |u-c_Q| >\ln(t)/c'\}\right)\;dt \\
&\leq l(Q)^\beta + \int_1^\infty \mathcal{H}^{\beta}_{\infty} \left(\{x \in Q : |u-c_Q| >\ln(t)/c'\}\right)\;dt \\
&\leq l(Q)^\beta +  C l(Q)^\beta \int_1^\infty \exp\left(\frac{-c\ln(t)}{c'}\right)\;dt \\
&=l(Q)^\beta \left(1+ C \int_1^\infty \frac{1}{t^{c/c'}}\right).
\end{align*}
In particular, the claimed estimate holds for any $c'<c$ and with constant
\begin{align*}
C'=\left(1+ C \int_1^\infty \frac{1}{t^{c/c'}}\right).
\end{align*}
\end{proof}

We next give the proof of Corollary \ref{pseminorms}.

\begin{proof}
One direction is straightforward:  The pseudo-H\"older inequality satisfied by the Hausdorff content \cite[(C7) on p.~5]{H1}, implies
\begin{align*}
\frac{1}{l(Q)^\beta} \int_Q |u-c| \;d\mathcal{H}^{\beta}_{\infty} &\leq \frac{C}{l(Q)^\beta} \left(\int_Q |u-c|^p \;d\mathcal{H}^{\beta}_{\infty}\right)^{1/p} \left(\int_Q |\chi_Q|^{p'} \;d\mathcal{H}^{\beta}_{\infty}\right)^{1/p'} \\
&= C \left(\frac{1}{l(Q)^{\beta}} \int_Q |u-c|^p \;d\mathcal{H}^{\beta}_{\infty}\right)^{1/p},
\end{align*}
from which the desired inequality follows from taking the infimum in $c$ and then the supremum over $Q \subset Q_0$.  For the other direction, Corollary \ref{exp_integrability} asserts
\begin{align*}
\int_Q \exp(c|u-c_Q|) \;d\mathcal{H}^{\beta}_{\infty} \leq C l(Q)^\beta
\end{align*}
for all $Q \subset Q_0$ and $u \in BMO^\beta(Q_0)$ with $\|u\|_{BMO^\beta(Q_0)} \leq 1$.
As for any $u \in BMO^\beta(Q_0)$, $u/\|u\|_{BMO^\beta(Q_0)}$ is admissible for application of this Corollary, from the Taylor expansion of the exponential, for any $u \in BMO^\beta(Q_0)$ we deduce
\begin{align*}
\int_Q c\frac{|u-c_Q|^p}{p!} \;d\mathcal{H}^{\beta}_{\infty} \leq C l(Q)^\beta \|u\|^p_{BMO^\beta(Q_0)} 
\end{align*}
The result then follows dividing by $l(Q)^\beta$, raising both sides to the $1/p$ power, and the observation that
\begin{align*}
(p!)^{1/p} \leq p.
\end{align*}
\end{proof}

We next prove Corollary \ref{nesting}.

\begin{proof}
Let $0<\alpha <\beta \leq d$.  Then for every $Q \subset Q_0$, with the choice of $c_Q$ which is suitable for $\mathcal{H}^{\alpha}_{\infty}$, using $\mathcal{H}^{\beta}_{\infty} \leq \left(\mathcal{H}^{\alpha}_{\infty}\right)^{\beta/\alpha}$ and the exponential decay proved in Theorem \ref{jn_content} we have
\begin{align*}
\int_Q |u-c_Q| \;d\mathcal{H}^{\beta}_{\infty} &= \int_0^\infty \mathcal{H}^{\beta}_{\infty}\left(\{ x\in Q : |u-c_Q|>t\}\right)\;dt \\
&\leq \int_0^\infty \left(\mathcal{H}^{\alpha}_{\infty}\left(\{ x\in Q : |u-c_Q|>t\}\right)\right)^{\beta/\alpha}\;dt \\
&\leq  C\int_0^\infty \left(l(Q)^\alpha C \exp(-ct/\|u\|_{BMO^\alpha(Q_0)})\right)^{\beta/\alpha}\;dt \\
&= C'l(Q)^\beta \|u\|_{BMO^\alpha(Q_0)}.
\end{align*}
As the infimum over $c$ is smaller than the left-hand-side, the result follows from dividing by $l(Q)^\beta$ and taking the supremum over finite subcubes $Q \subset Q_0$ parallel to $Q_0$.
\end{proof}

\section{Composition, Restriction, and a Sobolev Embedding}\label{otherresults}

We first prove Theorem \ref{composition}.

\begin{proof}
The fact that $\phi(0)=0$ implies
\begin{align*}
\int_{Q_0}  |\phi(u)| \;d\mathcal{H}^{\beta}_\infty \leq  \operatorname*{Lip}(\phi) \int_{Q_0}  |u| \;d\mathcal{H}^{\beta}_\infty,
\end{align*}
which along with the stability of quasicontinuity under composition with Lipschitz functions implies $\phi\circ u \in L^1(Q_0;\mathcal{H}^{\beta}_\infty)$.  Similarly, for every finite subcube $Q \subset Q_0$ parallel to $Q$ and any $c \in \mathbb{R}$ one has
\begin{align*}
\frac{1}{l(Q)^\beta}  \int_{Q}  |\phi\circ u-\phi \circ c| \;d\mathcal{H}^{\beta}_\infty \leq  \operatorname*{Lip}(\phi) \frac{1}{l(Q)^\beta}  \int_{Q}  |u-c| \;d\mathcal{H}^{\beta}_\infty,
\end{align*}
so that for every finite subcube $Q \subset Q_0$ parallel to $Q$ and suitable choice of $c \in \mathbb{R}$
\begin{align*}
\frac{1}{l(Q)^\beta}  \int_{Q}  |\phi\circ u-\phi \circ c| \;d\mathcal{H}^{\beta}_\infty \leq  \operatorname*{Lip}(\phi) \|u \|_{BMO^\beta(Q_0)}.
\end{align*}
But this shows
\begin{align*}
\inf_{c' \in \mathbb{R}} \frac{1}{l(Q)^\beta}  \int_{Q}  |\phi\circ u-c'| \;d\mathcal{H}^{\beta}_\infty \leq  \operatorname*{Lip}(\phi) \|u \|_{BMO^\beta(Q_0)}
\end{align*}
and therefore
\begin{align*}
\|\phi \circ u\|_{BMO^\beta(Q_0)} \leq \operatorname*{Lip}(\phi) \|u\|_{BMO^\beta(Q_0)}.
\end{align*}
\end{proof}

We next prove Theorem \ref{restriction}.

\begin{proof}
For $u \in BMO^k(Q_0)$ and any $k$-dimensional hyperplane $H_k$ parallel to $Q_0$, for suitable choice of $c>0$ the measure $\nu =c\mathcal{H}^k|_{Q_0\cap H_k} \in \mathcal{M}^k(Q_0)$ is admissible for the computation of the supremum.  Therefore, for any finite subcube $Q \subset Q_0$ parallel to $Q_0$, by the functional equivalence \eqref{HB}, one has the existence of $c_Q$ such that
\begin{align*}
\frac{1}{l(Q)^k} \int_{Q_0\cap H_k} \left|\;u|_{Q_0\cap H_k}- c_Q \right| d\mathcal{H}^k \leq C_\beta'\|u\|_{BMO^k(Q_0)}.
\end{align*}
But the left-hand-side lives above the infimum over all such $c \in \mathbb{R}$ so that
\begin{align*}
\sup_{Q \subset Q_0} &\inf_{c \in \mathbb{R}} \frac{1}{l(Q)^k} \int_{Q_0\cap H_k} \left|\;u|_{Q_0\cap H_k}- c\right| d\mathcal{H}^k \\
&\leq \sup_{Q \subset Q_0} \frac{1}{l(Q)^k} \int_{Q_0\cap H_k} \left|\;u|_{Q_0\cap H_k}- c_Q \right| d\mathcal{H}^k \leq C_\beta'\|u\|_{BMO^k(Q_0)}.
\end{align*}
It only remains to observe that the first line in the preceding is equivalent to the semi-norm introduced in \eqref{bmo_cube}, which is argued in the classical setting by John and Nirenberg just after \cite[(3)''']{JN}.  In particular, if we let $u_Q$ denote the average of $u$ over $\Pi_k (Q)$ with respect to $\mathcal{H}^k |_{\Pi_k (Q)}$, then
\begin{align*}
 \inf_{c \in \mathbb{R}} \frac{1}{l(Q)^k} \int_{Q_0\cap H_k} \left|\;u|_{Q_0\cap H_k}- c\right| d\mathcal{H}^k &\leq \frac{1}{l(Q)^k} \int_{Q_0\cap H_k} \left|\;u|_{Q_0\cap H_k}- u_Q \right| d\mathcal{H}^k \\
 &\leq \frac{1}{l(Q)^k} \frac{1}{l(Q)^k} \int_{Q_0\cap H_k} \int_{Q_0\cap H_k} \left|u(x)- u(y) \right| d\mathcal{H}^k(x) d\mathcal{H}^k(y) \\
 &\leq 2 \frac{1}{l(Q)^k} \int_{Q_0\cap H_k} |u(x)- c' | d\mathcal{H}^k(x)
\end{align*}
for any choice of $c' \in \mathbb{R}$.  The result follows from taking the infimum in $c'$ and then the supremum over subcubes $Q \subset Q_0$ parallel to $Q_0$.

In the case $k=d$, the equivalence of $\mathcal{H}^{d}_\infty$ and the Lebesgue measure implies
\begin{align*}
 \frac{1}{C} \frac{1}{l(Q)^d}  \int_{Q}  |u-c| \;d\mathcal{H}^{d}_\infty \leq   \frac{1}{l(Q)^d}  \int_{Q}  |u-c| \;dx \leq   C \frac{1}{l(Q)^d}  \int_{Q}  |u-c| \;d\mathcal{H}^{d}_\infty,
\end{align*}
from which the identity $BMO^d(Q_0) \equiv BMO(Q_0)$ follows from the same argument as above.
\end{proof}

We conclude the paper with a proof of Theorem \ref{morrey}.
\begin{proof}
We now show that $I_\alpha \mu \in BMO^{d-\alpha+\epsilon}(\mathbb{R}^d)$ for $\epsilon \in (0,\alpha)$. To this end, let $Q \subset \mathbb{R}^d$ be any finite cube and for $\mu \in \mathcal{M}^{d-\alpha}(\mathbb{R}^d)$ we let $2Q$ be the cube with the same center and twice the side length and define
\begin{align*}
\mu' &= \mu \chi_{2Q}\\
\mu''&= \mu - \mu'.
\end{align*}
Then for any $\nu \in \mathcal{M}^{d-\alpha+\epsilon}(\mathbb{R}^d)$ with $\|\nu\|_{\mathcal{M}^{d-\alpha+\epsilon}} \leq 1$ and any $c \in \mathbb{R}$ we have
\begin{align*}
 \frac{1}{l(Q)^{d-\alpha+\epsilon}}  \int_{Q}  |I_\alpha \mu-c| d\nu &\leq  \frac{1}{l(Q)^{d-\alpha+\epsilon}}  \int_{Q}  |I_\alpha \mu'| d\nu \\
 &\;\;+ \frac{1}{l(Q)^{d-\alpha+\epsilon}}  \int_{Q}  |I_\alpha \mu''-c| d\nu\\
 &=:I+II.
\end{align*}
Concerning $I$, we have the bound
\begin{align*}
I &\leq \frac{1}{l(Q)^{d-\alpha+\epsilon}} \frac{1}{\gamma(\alpha)} \int_{Q}  \int_{2Q} \frac{d|\mu|(y)}{|x-y|^{d-\alpha}} d\nu(x) \\
&= \frac{1}{\gamma(\alpha)} \frac{1}{l(Q)^{d-\alpha+\epsilon}} \int_{2Q} \int_{Q}   \frac{d\nu(x) }{|x-y|^{d-\alpha}} d|\mu|(y).
\end{align*}
However, a dyadic splitting yields 
\begin{align*}
\int_{Q}   \frac{d\nu(x) }{|x-y|^{d-\alpha}} &= \sum_{n=1}^\infty \int_{2^{-n+1}Q \setminus 2^{-n}Q}   \frac{d\nu(x) }{|x-y|^{d-\alpha}} \\
&\leq \sum_{n=1}^\infty  \frac{1}{(2^{-n}l(Q))^{d-\alpha}} \int_{2^{-n+1}Q \setminus 2^{-n}Q}   d\nu(x) \\
&\leq \sum_{n=1}^\infty  \frac{1}{(2^{-n}l(Q))^{d-\alpha}} \|\nu\|_{\mathcal{M}^{d-\alpha+\epsilon}} (2^{-n+1} l(Q))^{d-\alpha+\epsilon} \\
&= C\|\nu\|_{\mathcal{M}^{d-\alpha+\epsilon}}l(Q)^\epsilon.
\end{align*}
Therefore
\begin{align*}
I &\leq \frac{C}{\gamma(\alpha)} \|\nu\|_{\mathcal{M}^{d-\alpha+\epsilon}}\frac{1}{l(Q)^{d-\alpha}} \int_{2Q} d|\mu|(y) \\
&\leq C' \|\nu\|_{\mathcal{M}^{d-\alpha+\epsilon}} \|\mu\|_{\mathcal{M}^{d-\alpha}}.
\end{align*}

For $II$, the fact that $I_\alpha |\mu|(x)<+\infty$ for Lebesgue almost every $x \in \mathbb{R}^d$ yields the existence of $x_0 \in Q$ for which
\begin{align*}
c \vcentcolon= \int_{(2Q)^c} \frac{\mu(z)}{|x_0-z|^{d-\alpha}}<+\infty.
\end{align*}
With such a choice of $c$, we estimate
\begin{align*}
II &\leq   \frac{1}{\gamma(\alpha)}  \frac{1}{l(Q)^{d-\alpha+\epsilon}} \int_{Q} \int_{(2Q)^c}\left| \frac{1}{|x-z|^{d-\alpha}}-\frac{1}{|x_0-z|^{d-\alpha}}\right| d|\mu|(z) d\nu(x).
\end{align*}
The fact that $x, x_0 \in Q$ and $z \in (2Q)^c$ yields, by the mean value theorem, the estimate
\begin{align*}
\left| \frac{1}{|x-z|^{d-\alpha}}-\frac{1}{|x_0-z|^{d-\alpha}}\right| \leq \tilde{C} \frac{|x-x_0|}{|x_0-z|^{d-\alpha+1}},
\end{align*}
and therefore
\begin{align*}
II &\leq  \tilde{C}'l(Q)  \frac{1}{l(Q)^{d-\alpha+\epsilon}} \int_{Q}  \int_{(2Q)^c} \frac{1 }{|x_0-z|^{d-\alpha+1}} d|\mu|(z) d\nu(x).
\end{align*}
An analogous dyadic splitting to $I$ gives
\begin{align*}
\int_{(2Q)^c} \frac{d|\mu|(z) }{|x_0-z|^{d-\alpha+1}}   &= \sum_{n=1}^\infty \int_{2^{n+1}Q \setminus 2^{n}Q}   \frac{d|\mu|(z) }{|x_0-z|^{d-\alpha+1}}  \\
&\leq \sum_{n=1}^\infty  \frac{1}{(2^{n}l(Q))^{d-\alpha+1}} \int_{2^{n+1}Q \setminus 2^{n}Q}   d|\mu|(z)  \\
&\leq \sum_{n=1}^\infty  \frac{1}{(2^{n}l(Q))^{d-\alpha+1}} \|\mu\|_{\mathcal{M}^{d-\alpha}} (2^{n+1} l(Q))^{d-\alpha} \\
&= \frac{ \tilde{C}''}{l(Q)}\|\mu\|_{\mathcal{M}^{d-\alpha}},
\end{align*}
which when inserted in the inequality for $II$ yields that
\begin{align*}
II &\leq  C'' \|\nu\|_{\mathcal{M}^{d-\alpha+\epsilon}} \|\mu\|_{\mathcal{M}^{d-\alpha}}.
\end{align*}
Putting these two estimates together we obtain
\begin{align*}
 \frac{1}{l(Q)^{d-\alpha+\epsilon}}  \int_{Q}  |I_\alpha \mu-c| d\nu   \leq \left(C'+C''\right) \|\nu\|_{\mathcal{M}^{d-\alpha+\epsilon}} \|\mu\|_{\mathcal{M}^{d-\alpha}},
 \end{align*}
 which dividing by $\|\nu\|_{\mathcal{M}^{d-\alpha+\epsilon}}$ and taking the supremum over such measures, implies by the functional equivalence \eqref{HB}, that
\begin{align*}
 \frac{1}{l(Q)^{d-\alpha+\epsilon}}  \int_{Q}  |I_\alpha \mu-c| \;d\mathcal{H}^{d-\alpha+\epsilon}_\infty   \leq C''' \|\mu\|_{\mathcal{M}^{d-\alpha}}.
 \end{align*}
As the infimum over $c$ is smaller than the left-hand-side, the claim then follows by taking the supremum over cubes $Q$.  This completes the proof.
\end{proof}

\section*{Acknowledgments}
The authors would like to thank Rahul Garg, Augusto Ponce, Chun-Yen Shen, and Cody Stockdale for discussions concerning the results obtained in this manuscript.  Y.-W. Chen is supported by the National Science and Technology Council of Taiwan under research grant number 110-2115-M-003-020-MY3.  D. Spector is supported by the National Science and Technology Council of Taiwan under research grant number 110-2115-M-003-020-MY3 and the Taiwan Ministry of Education under the Yushan Fellow Program.  Part of this work was undertaken while D. Spector was visiting the IRMP Institute of the Universit\'e catholique de Louvain.  He would like to thank the IRMP Institute for its support and Augusto Ponce for his warm hospitality during the visit.

\end{document}